\title{A Geometric Proof of the Colored Tverberg Theorem}
\documentclass[11pt]{article}

\usepackage{a4wide}
\usepackage{amsmath,amssymb,amsthm, graphicx}
\usepackage{amsfonts}

\newtheorem{theorem}{Theorem}
\newtheorem{prop}[theorem]{Proposition}

\newtheorem{lemma}[theorem]{Lemma}

\theoremstyle{remark}

\newcommand{\R}{\mathbb{R}}
\newcommand{\RR}{\mathcal{R}}

\newcommand\bR{{\bf R}}
\newcommand{\C}{\mathcal{C}}
\newcommand{\PP}{\mathcal{P}}

\newcommand{\heading}[1]{\medskip\par\noindent{\bf #1}}

\DeclareMathOperator{\conv}{conv}

\DeclareMathOperator{\aff}{aff}
\DeclareMathOperator{\linaff}{linaff}
\DeclareMathOperator{\sgn}{sgn}
\DeclareMathOperator{\csgn}{csgn}
\DeclareMathOperator{\gsgn}{gsgn}

\def\:{\colon}
\def\eps{\varepsilon}

\long\def\onefigure#1#2{%  #1 picture,  #2  caption
\begin{figure*}[tbp]
\begin{center}
#1
\end{center}
\caption{#2}
\end{figure*}
}

\newcommand{\cmt}[1]{\ifhmode\newline\fi{\sf *** \ \ #1 \\}}

\newcommand{\labpdffig}[2]  % labeled PDF figure
{\onefigure{\mbox{\includegraphics{#1}}}{\label{f:#1} #2} }

\newcommand{\labepsfigw}[3]  % labeled EPS figure
{\onefigure{\mbox{\epsfig{file=#1.eps,width=#2}}}{\label{f:#1} #3} }

\makeatletter
\newcommand{\ProofEndBox}{{\ifhmode\unskip\nobreak\hfil\penalty50 \else
          \leavevmode\fi\quad\vadjust{}\nobreak\hfill$\Box$
            \finalhyphendemerits=0 \par}}
\makeatother

\def\kamsymb{{\rm a}}
\def\itisymb{{\rm b}}
\def\ethsymb{{\rm c}}
\def\snfsymb{{\rm e}}
\def\gauksymb{{\rm d}}

\author{
{\sc Ji\v{r}\'{\i} Matou\v{s}ek}$^{\kamsymb,\itisymb, \ethsymb}$
\and
{\sc Martin Tancer}$^{\kamsymb,\itisymb,\gauksymb}$
\and
{\sc Uli Wagner}$^{\ethsymb,\, \snfsymb}$
}

\begin{document}

\maketitle

{\renewcommand\thefootnote{\kamsymb}
\footnotetext{Department of Applied Mathematics,
Charles University, Malostransk\'{e} n\'{a}m.~25,
118~00~~Praha~1,  Czech Republic}
}
{\renewcommand\thefootnote{\itisymb}
\footnotetext{Institute of Theoretical Computer Science (ITI),
Charles University, Malostransk\'{e} n\'{a}m.~25,
118~00~~Praha~1,  Czech Republic}
}
{\renewcommand\thefootnote{\ethsymb}
\footnotetext{Institute of  Theoretical Computer Science,
ETH Zurich, 8092~Zurich, Switzerland}
}
{\renewcommand\thefootnote{\gauksymb}
\footnotetext{Supported by the grants SVV-2010-261313
(Discrete Methods and Algorithms) and GAUK 49209.}
}
{\renewcommand\thefootnote{\snfsymb}
\footnotetext{Research supported by the Swiss National Science Foundation (SNF Projects 200021-125309 and 200020-125027.}
}

\begin{abstract} 
The colored Tverberg theorem asserts that for every
$d$ and $r$ there exists $t=t(d,r)$ such that for every
set $C\subset \R^d$ of cardinality $(d+1)t$, partitioned
into $t$-point subsets $C_1,C_2,\ldots,C_{d+1}$ (which we
think of as color classes; e.g., the points of $C_1$ are
red, the points of $C_2$ blue, etc.),
there exist $r$ disjoint sets $R_1,R_2,\ldots,R_r\subseteq C$
that are \emph{rainbow}, meaning that $|R_i\cap C_j|\le 1$
for every $i,j$, and whose convex hulls all have a common point.

All known proofs of this theorem are topological.
We present a geometric version of a recent beautiful proof
by Blagojevi\'c, Matschke, and Ziegler, avoiding a direct use
of topological methods. The purpose of this de-topologization
is to make the proof more concrete and intuitive,
and accessible to a wider audience. 
\end{abstract}


\section{Introduction}

We first recall three fundamental
results of discrete geometry, all of them dealing
with partitioning finite sets in $\R^d$
so that the convex hulls of the parts intersect.
In the order of increasing sophistication, they are \emph{Radon's lemma},
\emph{Tverberg's theorem}, and the \emph{colored Tverberg theorem}.
We refer to \cite{Mat-dg} for more background, applications,
and historical references not mentioned here.

\heading{Radon's theorem }
 asserts that every set $C\subset\R^d$ of $d+2$ points
has two disjoint subsets $A_1,A_2$ with $\conv(A_1)\cap\conv(A_2)\ne\emptyset$;
see the illustration of the planar case in Fig.~\ref{f:radtve}.
The proof is simple linear algebra. 
%and $d+2$ is the smallest
%possible number, since a set $C$ of $d+1$ affinely independent points
%clearly has no two disjoint subsets with intersecting convex hulls.

\heading{Tverberg's theorem } states that every set $C\subset\R^d$ of
$(d+1)(r-1)+1$ points has $r$ pairwise disjoint subsets
$A_1,\ldots,A_r$ with $\bigcap_{i=1}^r\conv(A_i)\ne\emptyset$
(so Radon's lemma is the $r=2$ case). Several geometric proofs
are known, e.g., \cite{Tverberg,Tverberg2,sarkaria92}.
The number $(d+1)(r-1)+1$ is easily shown to be the smallest possible
for such a claim to hold,
e.g., by considering the configuration $C$ of $(d+1)(r-1)$ points
forming $d+1$ small clusters by $r-1$ points each, as in 
Fig.~\ref{f:clusters}. 

It is easy to show (e.g.,  by iterating  Radon's lemma) 
that there exists \emph{some}
number $T=T(d,r)$ such that the conclusion of the theorem
holds for every set $C$ with at least $T$ points.
The hard part of Tverberg's
theorem is obtaining the optimal value of~$T(d,r)$.

\heading{The colored Tverberg theorem } has a setting
similar to that of Tverberg's theorem. Again we have
a set $C\subset\R^d$ and seek $r$ pairwise disjoint subsets
whose convex hulls all share a point, but this time these subsets 
have to satisfy an additional restriction.

We introduce the following terminology. Let $C\subset\R^d$ be
a finite set  partitioned into $k$ \emph{color classes}
$C_1,C_2,\ldots,C_{k}$ (in other words, each point of $C$
is colored by one of $k$ colors). A subset $R\subseteq C$
is \emph{rainbow} if it contains at most one
point of each color, i.e., $|R\cap C_j|\le 1$ for all~$j$.

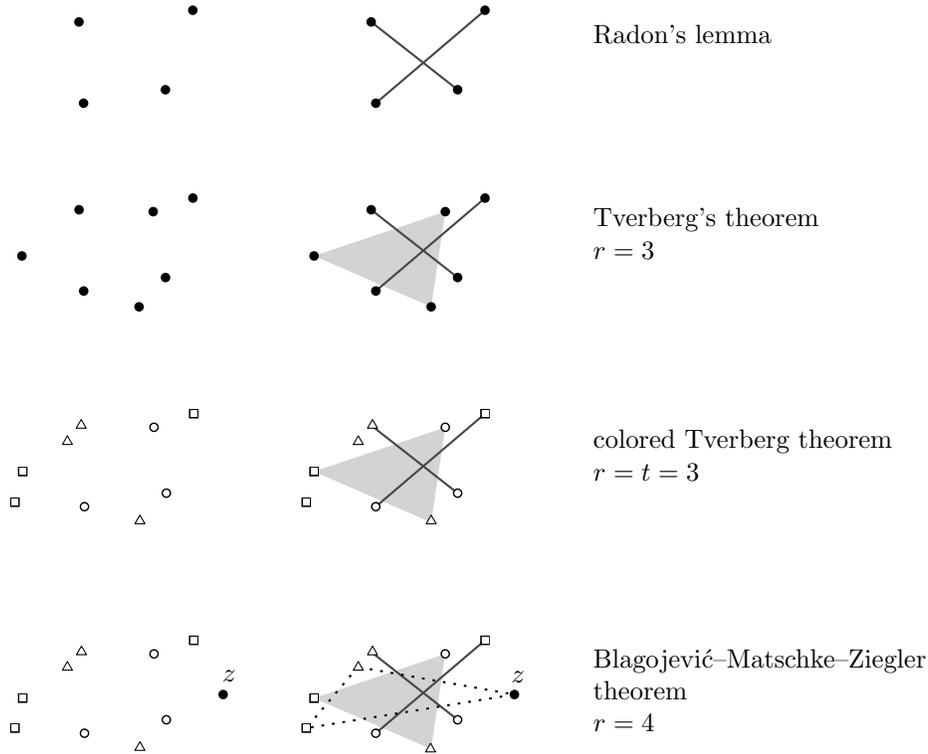
\labpdffig{radtve}{Radon's lemma, Tverberg's theorem,
the colored Tverberg theorem, and the Blagojevi\'c--Matschke--Ziegler
theorem: planar illustrations.}

\labpdffig{clusters}{A configuration with no Tverberg $r$-partition. }

A \emph{rainbow $r$-partition} for $C$ is an ordered
$r$-tuple $\RR=(R_1,\ldots,R_r)$ of pairwise disjoint rainbow subsets of $C$.
We stress that, for technical convenience, and with a mild abuse of the terminology ``partition'', 
we generally do \emph{not} require that the $R_i$ cover all of 
$C$ (if they do, we speak of a \emph{maximal rainbow $r$-partition}). 

A rainbow $r$-partition is \emph{Tverberg} if it has
a \emph{Tverberg point}, i.e., a point $x\in\bigcap_{i=1}^r\conv(R_i)$
(which usually does not belong to~$C$). The colored Tverberg theorem
can then be stated as follows.

%A particular case of the colored Tverberg theorem, illustrated
%in Fig.~\ref{f:radtve}, asserts that for every choice
%of $3$ red, $3$ blue, and $3$ green points in the plane
%there exists a (maximal) Tverberg rainbow $3$-partition,
%i.e., three red-blue-green triangles with disjoint vertex sets
%but with a point in common. The general case is as follows.

\begin{theorem}[Colored Tverberg theorem]\label{t:ctve}
 For every $d\ge 1$ and $r\ge 2$
there exists $t$ such that whenever $C\subset\R^d$
is a set of $(d+1)t$ points partitioned into
$t$-point subsets $C_1,\ldots,C_{d+1}$, then there is
a Tverberg rainbow $r$-partition $\RR=(R_1,\ldots,R_r)$ for~$C$.
\end{theorem}

The theorem is usually stated with $\RR$  \emph{maximal},
in which case each $R_i$
has to be a $(d+1)$-element set containing one point of each color.
However, we chose to omit maximality, since on the one hand, the proof
typically does not yield a maximal $\RR$, and on the other hand,
some thought reveals that, in the situation of Theorem~\ref{t:ctve},
an arbitrary $\RR$ can easily be extended
into a maximal one.

For the colored Tverberg theorem, proving the existence
of any $t$, no matter how large, seems difficult, and the
simplest proof currently known is also the one that yields
the smallest $t$, as we will briefly discuss below.

Let $t(d,r)$ denote the smallest $t$ for which the conclusion
of the theorem holds.
%\footnote{Naturally, one may also consider
%color classes of unequal sizes
 The configuration with $d+1$ clusters by $r-1$
points each, as in
Fig.~\ref{f:clusters}, where the $i$th cluster is all colored
with color $i$, shows that $t(d,r)\ge r$.

\heading{Historical notes. }
The validity of the colored Tverberg theorem was first conjectured
by B\'ar\'any,
F\"uredi and Lov\'asz~\cite{barany-furedi-lovasz90}, who proved
the case $d=2$, $r=3$, obtaining $t(2,3)\le 7$.
B\'ar\'any and Larman~\cite{barany-larman92} 
settled the planar case, showing $t(2,r)=r$ for all $r$
(their paper also contains Lov\'asz' topological 
proof showing that $t(d,2)=2$ for all $d$).
They conjectured that $t(d,r)=r$ for all $r,d$.

The first proof of the general case of the colored Tverberg theorem
was obtained by  \v Zivaljevi\'c and Vre\'cica~\cite{zivaljevic-vrecica92}
(simpler versions were provided in \cite{bjorner-at-al94,matousek96}).
Their proof is topological, and it builds on the pioneering
works by Bajm{\'o}czy and B{\'a}r{\'a}ny
\cite{BajmoczyBarany} (who gave a new, topological proof
of Radon's lemma) and by
B\'{a}r\'{a}ny, Shlosman, and Sz{\H{u}}cs \cite{BaranySS}
(who provided a topological proof of Tverberg's theorem assuming that
$r$ is a prime number). 

The \v Zivaljevi\'c--Vre\'cica method yields
$t(d,r)\le 2r-1$ for all \emph{prime} $r$. Later, the same bound
was extended to all $r$ that are \emph{prime powers}
\cite{zival-ugII}, using more advanced topological tools
introduced to combinatorial geometry by \"Ozaydin, by Volovikov,
and by Sarkaria.

The most important progress by far since the 1992 \v Zivaljevi\'c--Vre\'cica
proof was achieved by Blagojevi\'c, Matschke, and Ziegler
\cite{blagojevic-matschke-ziegler09arxiv} in 2009. 
They discovered a new proof, also topological, which yields
the optimal bound $t(d,r)=r$ \emph{whenever $r+1$ is a prime number}.

Their main trick is both simple and surprising; at first sight,
it seems strange that it might help in such a radical way.
Namely, to the point set $C=C_1\cup\cdots\cup C_{d+1}$
 as in the colored Tverberg theorem, with $|C_1|=\cdots=|C_{d+1}|=t$,
they first add an (arbitrary) extra point $z$, and color it
with a new color $d+2$, thus forming a singleton
color class $C_{d+2}=\{z\}$. Then they prove the existence
of a Tverberg rainbow $(r+1)$-partition $(R_1,\ldots,R_{r+1})$
for the set $C'=C_1\cup C_2\cup \cdots\cup C_{d+1}\cup C_{d+2}$.
Given such an $(r+1)$-partition, one can simply delete
the set $R_i$ containing the artificial point $z$, and be left
with a Tverberg $r$-partition for the original~$C$;
see the bottom part of Fig.~\ref{f:radtve}. 

We now formulate the main claim of the  Blagojevi\'c et al.\ proof.
With $r$ and $d$ fixed,
let us call a $(d+2)$-tuple $\C=(C_1,\ldots,C_{d+2})$ of
pairwise disjoint sets in $\R^d$ a
\emph{BMZ-collection} (BMZ standing for Blagojevi\'c--Matschke--Ziegler)
if $|C_1|=|C_2|=\cdots=|C_{d+1}|=r-1$ and $|C_{d+2}| = 1$.

\begin{theorem}[Blagojevi\'c--Matschke--Ziegler theorem]
\label{t:main} For every $d\ge 1$ and every \emph{prime} number $r$,
every BMZ-collection $\C$ admits a Tverberg rainbow $r$-partition
$(R_1,\ldots,R_{r})$.
\end{theorem}

We note that for proving the colored Tverberg theorem with $r=r_0$,
one uses the Blagojevi\'c--Matschke--Ziegler theorem with $r=r_0+1$.

Theorem~\ref{t:main} was first proved,
in a preliminary version of \cite{blagojevic-matschke-ziegler09arxiv},
 using relatively heavy topological machinery, 
by computing a certain obstruction in cohomology (this method
also yields additional results; see 
\cite{BMZ:OptimalBoundsTverbergVrecica-2011}).
Then Vre\'cica and
\v{Z}ivaljevi\'c~\cite{vrecica-zivaljevic09arxiv} found a simpler,
degree-theoretic proof, and independently, the authors
of~\cite{blagojevic-matschke-ziegler09arxiv} obtained a similar
simplification.

We should remark that the Blagojevi\'c--Matschke--Ziegler theorem
has a more general version, which is perhaps even nicer
and more natural.
Namely, for $r$ prime, whenever $N+1$ points in $\R^d$, $N=(d+1)(r-1)$,
are partitioned into classes $C_1,\ldots,C_m$, $m\ge d+1$, with
each $C_i$ of size
at most $r-1$, then there is a Tverberg rainbow $r$-partition.
This, for instance, also contains the original Tverberg theorem
as a special case. For simplicity, though, we will
consider only Theorem~\ref{t:main} in the rest of this paper.

\heading{This paper. } Our
main purpose is to present  an elementary and
 self-contained geometric proof of Theorem~\ref{t:main} 
(and thus of the colored Tverberg theorem as well).
We follow the basic strategy of the degree-theoretic
proof in \cite{vrecica-zivaljevic09arxiv,blagojevic-matschke-ziegler09arxiv}.
However, we replace the abstract \emph{deleted product construction}
by a concrete geometric construction due to Sarkaria \cite{sarkaria92}
(with a simplification by Onn). 

In this way, the basic scheme of the proof is clear and intuitive.
A rigorous elementary presentation avoiding topological tools is not entirely
simple, however, mainly because we have to deal with various issues 
of general position. These issues do not arise in the topological proof,
since they are dealt with on a general level when 
building the topological apparatus. 

One can say that our proof is a ``de-topologized'' version of the proofs in \cite{vrecica-zivaljevic09arxiv,blagojevic-matschke-ziegler09arxiv}.
In a similar sense,
Sarkaria's proof \cite{sarkaria92} of Tverberg's theorem
can be regarded as a de-topologized version of his earlier topological
proof of Tverberg's theorem \cite{Sarkaria-flores}. It has become
one of the most cited proofs, and often it is regarded as the 
standard proof, see, e.g., \cite[Section 1.2]{Kalai:CombinatoricsConvexity-95}, \cite[p. 30b]{Grunbaum:ConvexPolytopes2nd-2003}.\footnote{We remark that another possible strategy 
for de-topologizing Tverberg-type statements was suggested in \cite{Zivaljevic-InPursuitOfTheColoredCaratheodoryBaranyTheorems-1995}. The so-called \emph{guiding principle} 
on p.~94 of that paper suggests a certain way of relaxing the symmetry (equivariance)
 condition to obtain a more general geometric statement that might be more amenable to a purely 
 geometric proof in the spirit of Sarkaria's proof of the geometric Tverberg theorem.}
 
Another example of de-topologization is a combinatorial proof of Kneser's conjecture 
\cite{Mat-kne}; developing this approach further, Ziegler \cite{Ziegler-kne}
was able to prove all known generalizations of Kneser's conjecture,
plus some new ones, in a combinatorial way.

We hope that an elementary, de-topologized proof of the colored Tverberg theorem 
will stimulate further research by making the proof more intuitive and concrete and 
accessible to a wider audience. For example, this might help in attacking the open 
cases of the B\'ar\'any--Larman conjecture (the validity of the claim of Theorem~\ref{t:main}
for non-prime~$r$). 

While all known topological proofs of Tverberg's theorem work
only for $r$ that is a prime power, Sarkaria's de-topologized 
proof  \cite{sarkaria92} overcomes this restriction and works for all~$r$. 
Unfortunately, our de-topologization does not help in removing the 
restriction of prime $r$ in Theorem~\ref{t:main}. If anything, it helps 
in seeing more clearly why the proof method of Blagojevi\'c et al.~fails 
whenever $r$ is not a prime; see Section~\ref{s:concl} for a discussion.
%Nontheless, we hope that it may serve as a first step towards a fully combinatorial
%proof \`a la Sarkaria.

\section{Outline of the proof }

Here we sketch the main steps of the proof, proceeding informally
and glossing over many details.

We begin with a fixed BMZ-collection $\C=(C_1,\ldots,C_{d+1},C_{d+2}=\{z\})$.
We assume that the points of $\C$ are in a sufficiently general
position; if they are not, we use a standard perturbation argument.

We consider the system $\bR=\bR(\C)$
of all the \emph{maximal}
rainbow $r$-partitions $\RR=(R_1,\ldots,R_r)$ for $\C$,
Tverberg or not, for which $z\in R_r$.

Using a construction as in Sarkaria \cite{sarkaria92},
with each $\RR\in\bR$ we associate an $N$-dimensional
simplex $S_\RR$ in $\R^N$. (More precisely, some of the
$S_\RR$ may be degenerate, i.e., only $(N-1)$-dimensional,
even for $\C$ in general position, but this will not matter---so
for the purposes of this outline, we pretend that they are all
$N$-dimensional.)
The key property of this construction is  that
$\RR$ is Tverberg iff  $S_\RR$ contains the origin~$0$.

Moreover, all of the $S_\RR$ have one vertex $z^*$
in common. Let $F_\RR$ be the facet of $S_\RR$ opposite to $z^*$;
this is an $(N-1)$-dimensional simplex avoiding~$0$. 
Then we get that 
$\RR$ is Tverberg iff the ray $\rho$ emanating from $0$
in the direction opposite to $0z^*$ meets $F_\RR$; see Fig.~\ref{f:pseudoman}
for a schematic planar illustration.

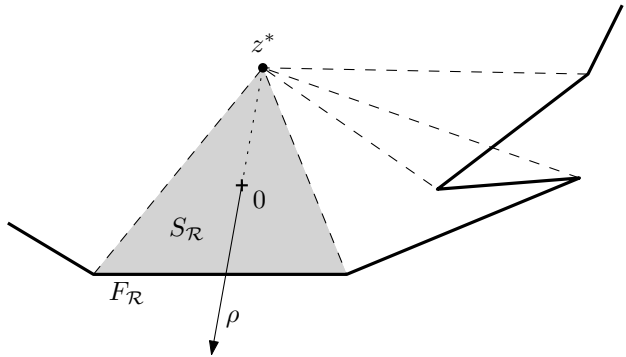
\labpdffig{pseudoman}{A schematic illustration of the situation
in $\R^N$.}

Next, it turns out that the union $\Sigma$ of all the $F_\RR$ forms
something like a (possibly self-intersecting) hypersurface in $\R^N$,
and one can define the \emph{degree} of $\Sigma$,
a standard notion in topology. (Since $\Sigma$ is determined by
$\C$, we also speak of the degree of $\C$ and write $\deg(\C)$.)

Intuitively, the degree counts how many times $\Sigma$ ``winds'' around $0$.
Its absolute value  is a lower bound for the number of times a ray like $\rho$
intersects $\Sigma$. Thus, if we can show that the degree of
$\Sigma$ is always nonzero, then $\rho$
has to intersect at least one $F_\RR$, and the existence of a Tverberg
rainbow $r$-partition follows.

%Defining the degree properly and checking consistency of the definition
%is the most technical part of the proof (here the standard topological
%machinery would be very helpful, but some computation
%is needed even in the topological setting). 

First we need to equip $\Sigma$ with
an \emph{orientation}, which means designating one of the
``sides'' of $\Sigma$ as positive and the other as negative;
see Fig.~\ref{f:pseudoman1}.
The orientation is defined locally: we determine positive
and negative side for every $F_\RR$, in a globally consistent
way. The definitions must match at the ``seams'' where two of the
$F_\RR$'s meet in an $(N-2)$-dimensional face.\footnote{From the
topological point of view, in this part we verify the well-known
fact (cf.~\cite{bjorner-at-al94})  that 
the abstract simplicial complex underlying $\Sigma$
is an \emph{orientable pseudomanifold}; this is a crucial
part of the proof, as well as of the proofs in
\cite{blagojevic-matschke-ziegler09arxiv,vrecica-zivaljevic09arxiv}.}  

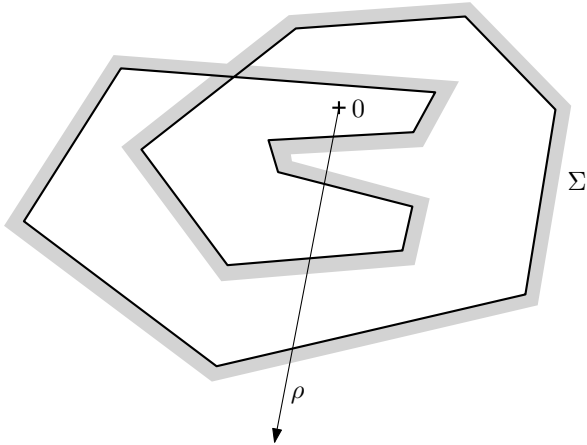
\labpdffig{pseudoman1}{Defining the degree of $\Sigma$;
the positive side of $\Sigma$ is marked gray.}

Then we define  the degree of $\Sigma$ as the number of times the ray $\rho$
passes from the negative side of $\Sigma$ to the positive side
minus the number of times it passes from the positive side
to the negative one (in the picture, the degree is $+2$).
As expected, the degree does not depend on the choice of
the ray $\rho$---any other ray emanating from $0$ yields the
same number. 
%(Some care is needed here, since the ray $\rho$
%might intersect $\Sigma$ in a degenerate way,
%i.e., in more than finitely many points.
%We will thus need to assume that $\C$ is in a ``sufficiently general''
%position, and check that then $\rho$ intersects $\Sigma$ nicely.
%In the topological proof, one need not worry about this,
%since there it is taken care of by the general machinery.)

It remains to verify that $\deg(\C)\ne 0$, and this is done
by a ``continuous motion'' argument. Namely, we fix a 
special BMZ-collection $\C_0$ for which the degree can be explicitly
computed. Then we consider
a continuous motion of the points of $\C_0$ that transforms
it to the given BMZ-collection $\C$. 
We follow the corresponding
motion of $\Sigma$ in $\R^N$ and look what happens to
its degree. It can change only when some of the $F_\RR$ pass through $0$.

We divide our collection $\bR$ of rainbow $r$-partitions
into classes of an equivalence $\sim$, where $\RR\sim\RR'$ if
$\RR'$  can be obtained
from $\RR$ by permuting the $R_i$ and, if needed,
moving $z$ back to the $r$th class. For example,
$$
\RR=(R_1,R_2,R_3)\sim 
\RR'=(R_2,R_3\setminus\{z\},R_1\cup\{z\}).
$$
Each class has $r!$ members, and it turns out that,
during the continuous motion, the simplices $F_\RR$
for all $\RR$ in the same class always pass through $0$
simultaneously, and their contributions
to the degree change by the same amount.

It follows that $\deg(\C)$ may change only by multiples
of $r!$ during the motion. Since the degree for the
special BMZ-collection $\C_0$ comes out as $D_0=\pm((r-1)!)^d$,
the degree for every $\C$ is congruent to $D_0$ modulo $r!$.

Here, finally, the primality of $r$ comes into play.
When $r$ is a prime, and only then, we have $D_0\not\equiv 0
\,({\rm mod}\, r!)$, and hence the degree is always nonzero as needed.

On the other hand, there are non-prime $r$ for which  BMZ-collections
$\C$ exist with degree $0$,  so indeed the proof method
breaks down (we suspect that this is the case for
\emph{all} non-prime $r$, but we have no proof 
at present). Of course, if the claim of the Blagojevi\'c--Matschke--Ziegler
theorem failed for some (non-prime) $r$, one would have to look for
a counterexample among the $\C$ with degree~$0$.
% However,
%our construction does not seem immediately
%useful for obtaining such a counterexample.

\section{The Sarkaria--Onn transform}
\label{s:so}

We start filling out the details in the above outline.
First we introduce the construction that assigns a point set in $\R^N$
to every rainbow $r$-partitions of a given BMZ-collection.
We present it in a slightly more general setting, ignoring
the ``rainbow'' aspect.

We will use the notation $[k]=\{1,2,\ldots,k\}$ for a positive
integer $k$.

For a point $x\in\R^d$ we write
$x^+$ for the vector $(x,1)\in\R^{d+1}$ obtained
by appending the component $1$ to~$x$. 

Let $w_1,\ldots,w_r$ be vectors in $\R^{r-1}$ forming the vertex
set of a regular $(r-1)$-dimensional simplex with center at the origin;
Fig.~\ref{f:w1wr} illustrates the case $r=3$.
We have $w_1+w_2+\cdots+w_r=0$.\footnote{The easiest way to see this to represent the regular $(r-1)$-simplex as the convex hull of the $r$ standard basis vectors in $\R^{r}$. Then we can identify $\R^{r-1}$ with the hyperplane $\{x\in \R^r\colon \sum_{i=1}^r x_i=1\}$, and choose a coordinate system such that the origin lies at the barycenter of the simplex, i.e., vector with all coordinates equal to $1/r$.} Moreover, if $\alpha_1,\ldots,\alpha_r$
are real numbers with $\alpha_1w_1+\cdots+\alpha_rw_r=0$, we have
$\alpha_1=\alpha_2=\cdots=\alpha_r$, 
since every $r-1$ of the
$w_i$ are linearly independent.

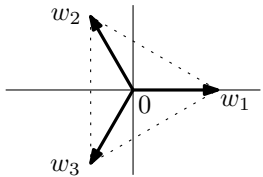
\labpdffig{w1wr}{The vectors $w_1,\ldots,w_r$ for $r=3$.}

For $x\in \R^d$ and an index $i\in[r]$, we
define a point 
$$
\varphi_i(x):= x^+ \otimes w_i\in \R^N,
$$
called the \emph{$i$th clone} of~$x$.
Here $N=(d+1)(r-1)$, and $\otimes$ stands for the (standard) tensor product:
for arbitrary vectors $u\in\R^m$ and $v\in\R^n$,
$u\otimes v$ is the vector 
$$(u_1v_1,u_1v_2,\ldots,u_1v_n,u_2v_1,u_2v_2,\ldots,u_mv_n)\in\R^{mn}.
$$

Now let $\PP=(P_1,P_2,\ldots,P_r)$ be an $r$-partition in $\R^d$, i.e.,
an $r$-tuple of pairwise disjoint finite sets in $\R^d$
(but the disjointness will be used only for a convenient
notation; the claims below remain valid even if the $P_i$
may share points). Let $P=P_1\cup\cdots\cup P_r$ be the ground set.

We define the \emph{Sarkaria--Onn transform} of $\PP$ as the
point set
$$
\Phi(\PP):=\bigcup_{i=1}^r \{\varphi_i(p): p\in P_i\},
$$
and we let
$$
S_\PP :=\conv(\Phi(\PP)).
$$
In words, for every point $p\in P_i$ we put  the $i$th clone 
of~$p$ in~$\Phi(\PP)$.

The following lemma is essentially from \cite{sarkaria92}.

\begin{lemma}[Sarkaria--Onn]
\label{l:sarkaria}
Let $\PP$ be an $r$-partition in $\R^d$. Then $\PP$
has a Tverberg point, i.e., satisfies $\bigcap_{i=1}^r\conv(P_i)\ne\emptyset$,
if and only if $0 \in S_{\PP}$.
\end{lemma}

\begin{proof} For the reader's convenience,
we sketch a proof;  the omitted details  are  easy to fill in.

First, let us suppose that $x\in \bigcap_{i=1}^r\conv(P_i)$
is a Tverberg point. Thus, for every $i$ we can write
$x=\sum_{p\in P_i}\xi_p p$ for some nonnegative reals
$\xi_p$ with $\sum_{p\in P_i}\xi_p=1$.
Then it is easy to check that 
$$
0=\frac1r\sum_{i=1}^r \sum_{p\in P_i} \xi_p \varphi_i(p)
$$
holds, and that this expresses $0$ as a convex combination of
the points of $\Phi(\PP)$.

Conversely, let us suppose that $0\in S_\PP$. Thus, we can write
\begin{equation}\label{e:tens}
0=\sum_{i=1}^r\sum_{p\in P_i} \alpha_p (p^+\otimes w_i) =
\sum_{i=1}^r \biggl(\sum_{p\in P_i} \alpha_p p^+\biggr)
\otimes w_i
\end{equation}
for some nonnegative $\alpha_p$'s summing to~$1$. Let $A_i:=
\sum_{p\in P_i} \alpha_p$ and $s_i:=\sum_{p\in P_i}\alpha_p p$.
By (\ref{e:tens}) we have $\sum_{i=1}^r A_iw_i=0$, and
so, by the properties of the $w_i$, 
all the $A_i$ are equal to some $A$. Similarly,
all the $s_i$ equal some $s\in\R^d$. Finally, one easily
checks that $A>0$ (since not all of the $\alpha_p$ are $0$)
 and that the point $\frac1A s$ is a Tverberg
point.
\end{proof}

In our considerations, we will need to interpret some other 
properties of $\Phi(\PP)$ in terms of $\PP$. 
We recall that the \emph{affine hull} $\aff(X)$ 
of a (finite) set $X\subseteq\R^d$ is the smallest
affine subspace of $\R^d$ containing $X$.
We also define the \emph{linear affine hull}
$\linaff(X)$ as the translation of $\aff(X)$ to $0$,
or in other words, as the set of all linear combinations
$\sum_{i=1}^n\beta_i x_i$ with $x_1,\ldots,x_n\in X$
and $\sum_{i=1}^n\beta_i=0$.

Let us say that the partition $\PP$ has an
\emph{affine Tverberg point} if
$\bigcap_{i=1}^r \aff(P_i)\ne\emptyset$.
Let us say that $\PP$ has  a
  \emph{Tverberg direction}
if $\bigcap_{i=1}^r\linaff(P_i)\ne \{0\}$; in other words, if there
is a line parallel to each of the $\aff(P_i)$. 

\begin{lemma}\label{l:dictio} For an $r$-partition $\PP$ in $\R^d$,
we have the following equivalences:
\begin{enumerate}
\item[\rm(i)] $0\in \aff(\Phi(\PP))$ iff 
 $\PP$  has an affine Tverberg point.
\item[\rm(ii)] The set $\Phi(\PP)$ is affinely dependent
iff at least one of the $P_i$'s
 is affinely dependent or $\PP$ has a Tverberg direction.
\end{enumerate}
\end{lemma}

\begin{proof} The proof is very similar to that of Lemma~\ref{l:sarkaria}
and we only sketch it, leaving the details to the interested reader.

In (i), the assumption $0\in \aff(\Phi(\PP))$ can be 
written as $\sum_{p\in P}\alpha_p p=0$ for some $\alpha_p$'s 
with $\sum_{p\in P}\alpha_p=1$. As in the proof of Lemma~\ref{l:sarkaria}, 
$\sum_{p\in P}\alpha_p p=0$ implies that the sums $\sum_{p\in P_i}\alpha_p$, $i\in[r]$, are all equal to the same
number $A$ and the sums $\sum_{p\in P_i}\alpha_p p$
are all equal to the same~$s$. From $\sum_{p\in P}\alpha_p\ne 0$
we get $A\ne 0$, and thus
$\frac1A s\in\aff(P_i)$ for all~$i$. 
The reverse implication in (i) is proved by going through
a very similar argument backwards.

As for (ii), we assume that the points of $\Phi(\PP)$ are affinely
dependent, i.e., there exist reals $\alpha_p$, $p\in P$, summing to $0$
and not all zero such that $\sum_{p\in P}\alpha_p p=0$. 
We again have $\sum_{p\in P_i}\alpha_p=A$
and $\sum_{p\in P_i}\alpha_p p=s$ for all $i\in[r]$. 
Since $\sum_{p\in P}\alpha_p=0$,
we get $A=0$.  If $s=0$, then at least one of the $P_i$ is affinely dependent,
and otherwise, $s$ is a nonzero vector in $\bigcap_{i=1}^r\linaff(P_i)$.
Again we omit the reverse implication.
\end{proof}

\section{Sufficiently general position}

\heading{Some conventions for BMZ-collections. }
Now we specialize to BMZ-collections. 
Let $\C=(C_1,\ldots,C_{d+2})$ be a BMZ-collection, and let us
write $C:=C_1\cup C_2\cup\cdots\cup C_{d+2}$ for its ground set.

We also assume that the points of $C$ are numbered
as $c_1,c_2,\ldots,c_{N+1}=z$, in such a way that $C_1$
consists of  the first $r-1$ points $c_1,\ldots,c_{r-1}$,
$C_2$ consists of the next $r-1$ points, etc.

Let $\RR$ be a rainbow $r$-partition for $\C$.
We define the \emph{combinatorial type} of $\RR$
as the set $\{(i,j): c_j\in R_i\}\subseteq [r]\times [N+1]$.

As in the proof outline, let $\bR$ be the collection
of all the maximal rainbow $r$-partitions having the point
$z$ in the last class. 

For $\RR=(R_1,\ldots,R_r)\in\bR$ and a point $a\in C$, we write $\RR-a$
for the rainbow $r$-partition $(R_1\setminus\{a\},\ldots,R_r\setminus\{a\})$
(we remove $a$ from the class it belongs to).

For every $\RR\in\bR$, we have $z\in R_r$, and so
each $\Phi(\RR)$ contains the point
 $z^*=\varphi_r(z)$. We set $F_\RR :=\conv(\Phi(\RR-z))$;
if $S_\RR$ is an $N$-dimensional simplex, which is usually the case,
then $F_\RR$ is the facet opposite to $z^*$ as in the outline.

\heading{Sufficiently general position. }
For defining the degree as sketched in the 
outline, we need that the simplices $F_\RR$ are in a suitably general position.
We adopt a ``functional'' approach, postulating the required properties
in a definition.

We say that $\C$ is in a \emph{sufficiently general position}
if
\begin{itemize}
\item each $F_\RR$ is an $(N-1)$-dimensional simplex,
i.e., its vertices are affinely independent, and
\item  for every $\RR\in\bR$ and every $a\in C$
we have $0\not\in\aff(\Phi(\RR-a))$; geometrically, the affine
span of each facet of $S_\RR$ avoids~$0$.
\end{itemize}

It is easily seen that for $\C$ in sufficiently general position,
the ray $\rho$ as in the outline
(emanating from $0$ in the direction opposite to $0z^*$)
is well defined and  intersects each $F_\RR$ in at most one point, which
lies in the relative interior of~$F_\RR$.

Let $\C,\C'$ be two BMZ-collections. We define their \emph{distance}
in the natural way, as $\max\{\|c_i-c'_i\|:i=1,2,\ldots,N+1\}$
where $\|.\|$ is the Euclidean norm and $c'_i$ is, of course,
the $i$th point of $\C'$.

We want to show that for every BMZ-collection $\C$, there are
BMZ-collections $\C'$ in sufficiently general position arbitrarily
close to it. We proceed by a standard perturbation argument
(an alternative route would be using points with algebraically
independent coordinates in $\C'$). This is a technical and somewhat
tedious part (in the topological proof, it is taken care
by the general machinery, so one need not worry about it).
Still, we prefer to include it, in order to make the proof complete.

\begin{lemma}\label{l:sgp}
Let $\C$ be a BMZ-collection, and let $\eps>0$ be given. Then there
is a BMZ-collection $\C'$ in sufficiently general position at distance
at most $\eps$ from~$\C$.
\end{lemma}

\begin{proof} First we observe that for $\RR\in\bR$, since
the classes $R_1,\ldots,R_r$ are rainbow,
each of the classes $R_i$ has at most $d+1$ points,
except possibly for $R_r$, which may contain up to $d+2$ points.

According to Lemma~\ref{l:dictio}, the conditions in the definition
of sufficiently general position of $\C$ are implied by the following:
\begin{enumerate}
\item[(i)] Every at most $d+1$ points of $C$ are affinely independent.
\item[(ii)] For every $\RR\in\bR$, the partition $\RR-z$ has no
Tverberg direction.
\item[(iii)] For every $\RR\in\bR$ and every $a\in C$, 
the partition $\RR-a$ has no affine Tverberg point.
\end{enumerate}

To seasoned geometers, (i)--(iii) are probably obvious by
codimension count. Still, we include a more detailed argument.

We first recall a perturbation argument for achieving (i),
where it is entirely simple and standard.
Condition (i) is a conjunction of ${|C|\choose d+1}$ requirements
of the form ``the points in $C_I:=\{c_i:i\in I\}$ are affinely independent'',
where $I$ runs through all $(d+1)$-element subsets of $C$.
We enumerate all such $I$ as $I_1,I_2,\ldots$ and we
deal with them one by one. 

First we consider $I_1$; say that
$I_1=\{1,2,\ldots,d+1\}$.
The one-point set $\{c_{1}\}$ is affinely independent, of course,
and so is $\{c_{1},c_{2}\}$, assuming that the points of $C$
are all distinct. Next, it is clear that we can move 
$c_{3}$ by at most $\frac\eps 2$ so that 
$C_3:=\{c_{1},c_{2},c_{3}\}$ is affinely independent, too.
Then we successively move $c_4,\ldots,c_{d+1}$, each by at most
$\frac\eps2$, and we we make $C_{I_1}$ affinely independent.
Moreover, crucially, there exists some $\eps_1>0$
such that if we move the points of $C_{I_1}$ arbitrarily by at most $\eps_1$,
then $C_{I_1}$ \emph{remains} affinely independent. Using this $\eps_1$, we make $C_{I_2}$
affinely independent, obtaining some even much smaller $\eps_2>0$,
etc., until all the index sets $I_j$ have been exhausted.

A similar procedure can be applied to achieve (ii) and (iii).
For example, in (iii), we fix $\RR$ and  $a$ and
see how can we make sure that $\RR-a$ has no  affine 
Tverberg point.
%\footnote{Strictly speaking, as we move the points
%of $\C$, the rainbow $r$-partitions move as well.
%So if we wanted to be formally precise, we should introduce
%new symbols for the moved collections and the corresponding}

Let us write $R^-_i:=R_i\setminus\{a\}$.
Each of the subspaces $L_i:=\aff(R_i)$ has dimension
at most $|R^-_i|-1$. 

In general, if two affine subspaces $K,L\subset\R^d$ of dimensions
$k,\ell$, respectively, are in general position,
we have $\dim(K\cap L)=\max(-1,k+\ell-d)$, where dimension $-1$
means empty intersection. 
Thus, we can move $L_2,L_3,\ldots,L_r$ one by one
(by moving the points of the $R_i^-$), inductively achieving
$\dim(L_1\cap\cdots \cap L_i)= \max(-1,(\sum_{j=1}^i|R_j^-|)-i-(i-1)d)$.
Since $\sum_{j=1}^r|R_j^-|=N=(r-1)(d+1)$, 
we get $\dim(L_1\cap\cdots \cap L_r)=-1$, which means no 
affine Tverberg point.

Condition (ii) is achieved with a very similar dimension-counting,
which we omit.
\end{proof}

\heading{A remark on degenerate $S_\RR$'s. }
The ``exceptional'' $S_\RR$'s that are
only $(N-1)$-dimensional, even for
$\C$ in sufficiently general position, are obtained for the $\RR$
with the last class $R_r$ of size $d+2$. Then $R_r$ cannot
be affinely independent, and thus (by Lemma~\ref{l:dictio})
the vertex set of $S_\RR$ is not affinely independent---the
point $z^*$ is contained in the affine span of $F_\RR$.
But this does not matter since, for $\C$ in sufficiently
general position, the affine span of $F_\RR$ avoids $0$
and thus such an $F_\RR$ cannot influence the degree.
(Or in other words, such a partition $\RR$ is never Tverberg
for $\C$ in sufficiently general position.)

\heading{Continuous motion of $\C$. }
Later on, in the continuous motion argument, we will need to consider
two BMZ-collections $\C,\C'$ and analyze what happens with the degree
when we continuously move the points, starting from $\C$ and 
ending at~$\C'$.

As we will see, the moving collection can be kept in sufficiently
general position all the time except for finitely many
\emph{critical times}. 

We will also need some control
of what happens at the critical times. Let $\RR\in\bR$ and let
$a\in C$, $a\ne z$. We call the set
$G:=\conv(\Phi(\RR-z-a))$ a \emph{ridge} (if $S_\RR$ is an $N$-simplex,
which is typically the case, then $G$ is a facet of $F_\RR$
and thus a ridge of $S_\RR$). We say that $\C$
is in \emph{almost general position} if all ridges avoid~$0$.

\begin{lemma}\label{l:motion-gen}
Let $\C,\C'$ be BMZ-collections in sufficiently general position. 
Then there is a continuous
family $\C^{(t)}$ of BMZ-collections, $t\in [0,1]$,
such that $\C^{(0)}=\C$,
$\C^{(1)}=\C'$, each $\C^{(t)}$ is in
almost general position, and there is a finite set
$T\subset [0,1]$ of critical times such that
$\C^{(t)}$ is in sufficiently general position
for all $t\not\in T$. 
\end{lemma}

\begin{proof} For simplicity, we move one point at a time.
It suffices to establish the lemma for $\C,\C'$ such that
 $c_i=c'_i$ for all $i\ne1$. Moreover, since all 
BMZ-collections sufficiently close to $\C'$ are 
also in a sufficiently general position, it is enough
that we can move $c_1$ to any position $c''_1$ sufficiently
close to $c'_1$, in a way satisfying the conclusion of the lemma,
since then the motion from $c''_1$ to $c'_1$ is for free.

Thus, from now on we assume that $c_1$ moves to $c''_1$
along a segment at uniform speed, while all the other points
are stationary. Let $c_1^{(t)}$ be the position of the
moving point at time~$t$.

First we check that there are only finitely many times
where $\C^{(t)}$ is not in sufficiently general position.
We need to consider conditions (i)--(iii) from the proof
of Lemma~\ref{l:sgp}. For the sake of illustration, we check (iii),
leaving the rest to the reader. 

Still referring to that proof,
we  consider the affine subspaces
$L_1,\ldots,L_r$ (for a particular $\RR$ and $a$).
 We renumber them so that the moving point
is among those defining $L_r$, so $L_1,\ldots,L_{r-1}$ are
stationary and $L_r^{(t)}$ is moving. Let $\overline L_r:=
\bigcup_{t\in[0,1]}L_r^{(t)}$; since $c_1^{(t)}$ traces
a segment, $\overline L_r$ is contained
in the affine span of $L_r\cup\{c''_1\}$, which is an affine subspace
of dimension  $\dim(L_r)+1$.
By sufficiently general
position of $\C$ we know that
$\dim(L_1\cap\cdots\cap L_{r-1})+\dim (L_r)<d$,
and thus, by altering the position of  $c''_1$ 
by an arbitrarily small amount, we
can achieve that $L_1\cap\cdots\cap L_{r-1}$ meets $\overline L_r$
in at most one  point. This adds at most one critical time.

It remains to check that $\C^{(t)}$ is always in almost general
position, for which the argument is very similar to
the previous one. We want that
all ridges avoid $0$ all the time. We strengthen the condition
to the \emph{affine} span of all ridges avoiding $0$, which
translates into $\RR-z-a$ never having an affine Tverberg point.
Thus, we again deal with affine subspaces $L_1,\ldots,L_r$;
we again assume that $L_1,\ldots,L_{r-1}$ are
stationary and $L_r^{(t)}$ is moving, and $\overline L_r$
be the set traced by $L_r^{(t)}$ during the motion, contained
in an affine subspace of dimension $\dim(L_r)+1$.
However, compared to the previous argument, now the
sum of the dimensions of the $L_i$ is one smaller,
 and this allows us to achieve
$L_1\cap\cdots\cap L_{r-1}\cap \overline L_r=\emptyset$,
again by changing the position of $c''_1$ by an arbitrarily
small amount. 
\end{proof}

\heading{Sufficiently general position may be assumed. }
We will prove Theorem~\ref{t:main} with the additional assumption
that $\C$ is in sufficiently general position.
By Lemma~\ref{l:sgp},
each BMZ-collection can be approximated by such BMZ-collections
 arbitrarily closely, and thus
the validity of Theorem~\ref{t:main} for an arbitrary $\C$
follows 
by a routine limiting argument, which we omit
(see, for example, \cite[Lemma~2]{Tverberg} for a very similar one).

\section{The degree}

For every $(N-1)$-dimensional simplex $F_\RR$, we now 
define a \emph{sign} $\sgn(F_\RR)$ (often
we also write just $\sgn(\RR)$, since $F_\RR$ is
fully determined by $\RR$). 
In the language introduced in the proof outline, the sign $+1$
means that the side of $F_R$ visible from $0$ is negative,
and $-1$ means that it is positive.

The sign is the product of two factors, which we call
the \emph{geometric sign} $\gsgn(\RR)$ and the \emph{combinatorial sign}
$\csgn(\RR)$.

\heading{The geometric sign } is easy to define.
We set up the $N\times N$ matrix $M$ with the coordinates of the
$i$th vertex of $F_\RR$ (we recall that the points
in the ground set $C$ are numbered as $c_1,\ldots,c_{N+1}$,
which induces a linear ordering of the vertices of $F_\RR$), and we put
$$
\gsgn(\RR) :=\sgn \det(M).
$$

\heading{The combinatorial sign } is slightly more complicated.
We recall that the combinatorial type of $\RR$
is the set $\{(i,j): c_j\in R_i\}\subseteq [r]\times [N+1]$.
It can be depicted using an $r\times (N+1)$
array of squares, whose $i$th row corresponds to the sets $R_i$
of $\RR$ and whose $j$th column corresponds to the $j$th point
of $C$; see Fig.~\ref{f:rooks}. Then we place a rook
(chess figure) to each square $(i,j)$ with $c_j\in R_i$.

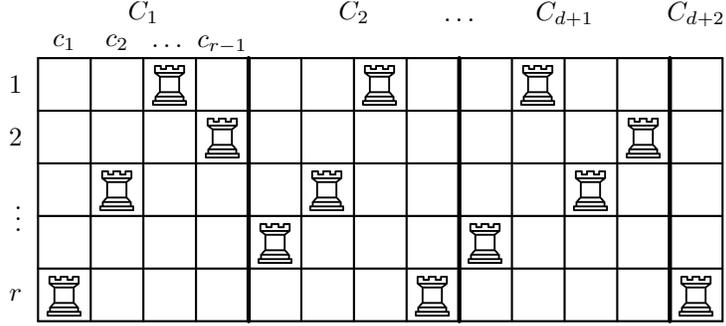
\labpdffig{rooks}{The combinatorial type of a rainbow $r$-partition
represented by a non-attacking placement of rooks on chessboards.}

Let us think of the array as $d+1$ chessboards, each with
$r$ rows and $r-1$ columns,
placed side by side, plus one ``degenerate'' $r\times 1$
chessboard on the right. Then 
the maximal rainbow $r$-partitions exactly correspond to
maximal placements of mutually non-attacking rooks on each
of the chessboards (in particular, each
of the $r\times(r-1)$ chessboards has $r-1$ rooks
on it). The condition that $z\in R_r$ then
says that the last narrow chessboard should have the rook
in the last row.

The vertices of $F_\RR$ correspond to the rooks
in the first $d+1$ chessboards. The placement of the $r-1$ rooks
on the $k$th chessboard 
defines a permutation $\pi_k$ of $[r]$; namely, for $j\le r-1$,
$\pi_k(j)$ is the index of the row containing the rook
of the $i$th column, and $\pi_k(r)$ is the index of the
unique row with no rook.

The combinatorial sign of $\RR$ is defined as
$$
\csgn(\RR):=\prod_{k=1}^{d+1}\sgn\pi_k.
$$

\heading{The degree. } 
As in the outline, we define 
$$\Sigma=\Sigma(\C):=\bigcup_{\RR\in\bR} F_\RR,
$$
and the degree   of $\Sigma$ is
$$
\deg(\Sigma):=\sum_{\RR\in\bR:\rho\cap F_\RR\ne\emptyset}
\sgn(F_\RR),
$$
where $\sgn(F_\RR)=\sgn(\RR)=\gsgn(\RR)\csgn(\RR)$.
In other words, the degree is the sum of $\sgn(F_\RR)$ over
all $\RR\in\bR$ that are Tverberg. Since $\Sigma$ is determined
by $\C$, we will also write $\deg(\C)$ instead of $\deg(\Sigma)$.

\section{The continuous motion argument}

Here we prove the promised invariance of the degree modulo $r!$.

\begin{prop}\label{p:deginvar}
If $\C$ and $\C'$ are two BMZ-collections (for the same $d$ and $r$),
both in sufficiently general position, then
$$
\deg(\C)\equiv \deg(\C')\,({\rm mod}\,r!).
$$
\end{prop}

First we want to verify that the simplices $F_\RR$ are ``glued together''
properly. Let us call the $(N-2)$-dimensional faces of $F_\RR$
the \emph{ridges}.

\begin{lemma} \label{l:pseudoman}
Let $G$ be a ridge of some $F_\RR$. Then there is
exactly one ${\RR'}\in\bR$ distinct from $\RR$ having $G$ as a ridge,
and we have $\csgn({{\RR'}})=-\csgn({\RR})$. (In topological
terminology, this is the 
``orientable pseudomanifold'' property.)
\end{lemma}

\begin{proof} This is easy to see using the rook interpretation.
The simplex $F_\RR$ corresponds to a maximal placement of
rooks on the first $d+1$ chessboards, and $G$ is obtained by 
removing one of the rooks, say from the $k$th chessboard. 
Now the $k$th chessboard has one empty column and two empty rows,
so there are two possibilities of putting the rook back---one
corresponding to $F_\RR$, and the other to $F_{{\RR'}}$.

The permutation $\pi_k$ for $\RR$ and the one for ${\RR'}$
differ by a single transposition, and so $\csgn(\RR)=-\csgn({\RR'})$
as claimed.
\end{proof}

Next, we want to see that the degree of $\Sigma$
can be computed with respect to an arbitrary (generic) ray.
Let $\C$ be a BMZ-collection, exceptionally assumed
to be only in almost general position
(which, as we recall, means that all the ridges of
the $F_\RR$'s avoid the origin).

Let $\psi$ be a ray in $\RR^N$ emanating from $0$.
We call $\psi$ \emph{generic} for $\C$ if 
it does not intersect any ridge.
It follows that if such a generic $\psi$
intersects some $F_\RR$, then $F_\RR$ must be
an $(N-1)$-dimensional simplex
and $\psi$ intersects it in a single point lying in the relative
interior of~$F_\RR$. 

Clearly, almost all rays (in the sense of measure) are generic. Moreover,
if $\psi$ is generic for some $\C$, then it is also generic for
all $\C'$ sufficiently close to $\C$; this will be useful later on.

Given a generic ray $\psi$, we define $\deg_\psi(\C)$
in the same way as we defined $\deg(\C)$ using $\rho$;
that is, as $\sum_{\RR\in\bR:\psi\cap F_\RR\ne\emptyset}\sgn(\RR)$.

\begin{lemma} \label{l:moveray}
Let $\C$ be a BMZ-collection in sufficiently general position.
If $\psi$ and $\nu$ are generic rays for $\C$,
then $\deg_\psi(\C)=\deg_\nu(\C)$.
\end{lemma}

\begin{proof} We can continuously move $\psi$ to $\nu$ so that
it remains generic all the time, except for finitely many
moments where it intersects some ridge (or perhaps several ridges) at
an interior point. So it suffices to check that the degree cannot
change by crossing a ridge~$G$.

As we know from Lemma~\ref{l:pseudoman}, the ridge $G$
is shared by exactly two facets $F_\RR$ and $F_{\RR'}$,
with $\csgn(\RR)=-\csgn({\RR'})$.
Let $v$ be the vertex of $F_\RR$ not in $G$,
and similarly for $v'$ and $F_{\RR'}$.
As we saw in the proof of  Lemma~\ref{l:pseudoman},
$v$ and $v'$ are two different clones of the same point $c_j\in C$.

Let $h$ be the hyperplane spanned by $G$ and $0$. First let us
suppose that both $v$ and $v'$ are at the same side of~$h$ (Fig.~\ref{f:crossG}
left). Then the moving ray intersects both of $F_\RR,F_{\RR'}$
before crossing $G$ and none of them after the crossing, or the other
way around. 

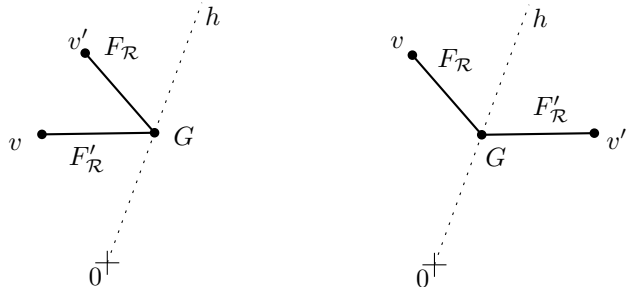
\labpdffig{crossG}{The moving ray crossing a ridge. }

Let $M$ and $M'$ be the matrices used in the definition of the geometric
signs of $F_\RR$ and $F_{\RR'}$, respectively. They differ in a single
row, which is $v$ in $M$  and $v'$ in $M'$ (the row is in the same
position since $v$ and $v'$ are both clones of $c_j$). 
Since $v$ and $v'$ are on the same side of $h$, we have
$\sgn(\det M)=\sgn(\det M')$, and thus $F_\RR$ and $F_{\RR'}$
have the same geometric signs.

Altogether we get $\sgn(\RR)=-\sgn({\RR'})$, and thus the
when the ray intersects both of $F_\RR,F_{\RR'}$, their contributions
to the degree cancel out. By a similar argument, which we omit,
one gets that in the other case, as in  Fig.~\ref{f:crossG} right,
$\sgn(\RR)=\sgn({\RR'})$, and so in both case the degree
remains constant.
\end{proof}

Let $\RR\in\bR$ be a rainbow $r$-partition of a BMZ-collection
$\C$ (in sufficiently general position). 
For a permutation $\pi$ of $[r]$, let $\RR^\pi$ be the
rainbow $r$-partition obtained by permuting the classes of $\RR$
according to $\pi$ and moving $z$ back to the last class:
$$
\RR^\pi:=\left(R_{\pi(1)}\setminus \{z\},R_{\pi(2)}\setminus\{z\},\ldots,
R_{\pi(r-1)}\setminus\{z\}, R_{\pi(r)}\cup\{z\}\right).
$$
We need to understand how the combinatorial and geometric 
signs of $\R^\pi$ are related to those of $\RR$.

\begin{lemma}\label{l:permuting}
For $\RR$ and $\RR^\pi$ as above, we have
$$
\csgn({\RR^\pi})=\sgn(\pi)^{d+1}\csgn(\RR),
\ \ \ 
\gsgn({\RR^\pi})=\sgn(\pi)^{d+1}\gsgn(\RR).
$$
\end{lemma}

\begin{proof}
In the representation of $\RR$ with rooks, 
passing to $\RR^\pi$ means that we permute the rows of
the first $d+1$ chessboards. From this we immediately
get the first relation, $\csgn({\RR^\pi})=\sgn(\pi)^{d+1} \csgn({\RR})$.

For the geometric sign, it suffices to consider 
the case where $\pi$ is a transposition exchanging two indices
$i,j$ and show that the geometric sign changes by the factor
of $(-1)^{d+1}$ (an arbitrary $\pi$ can be expressed as a composition of such
transpositions).

The effect of such a transposition on the vertex set of $F_\RR$
is that the $i$th clones of the points of $R_i$ are replaced
with the $j$th clones, and the reverse happens for the points
of $R_j$ (ignoring $z$). 

Let $M$ be the matrix as in the definition of $\gsgn(\RR)$,
and let $M^\pi$ be the one for $\gsgn(\RR^\pi)$. Thus, a row
of the form $x^+\otimes w_i$ in $M$ is replaced by $x^+\otimes w_j$
in $M^\pi$. Similarly, $x^+\otimes w_j$ is replaced by
$x^+\otimes w_i$, and all other rows remain unchanged.

Now we use the choice of the vectors $w_1,\ldots,w_r$.
They form the vertex set
of a regular simplex, and so there is a linear map 
$f\:\R^{r-1}\to \R^{r-1}$ that interchanges $w_i$ with
$w_j$ and leaves all the other $w_k$ fixed (namely, $f$ is a suitable
mirror reflection). 

Let $A$ be the matrix of $f$ with respect to the standard basis of $\R^{r-1}$.
Then we can write $M^\pi = BM$, where
$B$ is the block-diagonal matrix with $d+1$ blocks $A$ on the diagonal.
Thus, $\det(M^\pi)=\det(A)^{d+1}\det(M)$, and since
$f$ is a mirror reflection, and thus orientation-reversing,
we have $\sgn(\det A)=-1$.
So the geometric sign changes by $(-1)^{d+1}$ as claimed.
\end{proof}

\begin{proof}[Proof of Proposition~\ref{p:deginvar}. ]
The main trick in the proof is to alternate moving the
ray and the points, thereby avoiding ``too degenerate'' situations.

Using Lemma~\ref{l:motion-gen}, we may assume that
$\C$ and $\C'$ are connected by a continuous family
$\C^{(t)}$.  Each $\C^{(t)}$ is in almost general
position, and it is in sufficiently general position 
except for finitely many critical times.

For every  $t\in [0,1]$, including critical ones,
we can choose a generic ray for $\C^{(t)}$,
which also remains generic for all $\C^{(t')}$ with $t'$ in some
open interval around $t$.
By compactness,
the interval $[0,1]$ can be covered by finitely many of these open
intervals $I_1,\ldots,I_m$, each of them corresponding to some generic ray
$\psi_i$. 

By Lemma~\ref{l:moveray}, on the overlapping part $I_i\cap I_j$
we can ``measure'' the degree using either $\psi_i$ or $\psi_j$
with the same result. Therefore, it suffices to show that
if $I\subseteq [0,1]$ is an interval such that $\psi$
is a generic ray for all $\C^{(t)}$ with $t\in I$, then 
$\deg_\psi(\C^{(t)})$ may change only by multiples of $r!$.

The degree may change only at critical values of $t$; let $t_0\in I$
be one of the critical values. Let us  see how the 
contribution of some $F_{\RR}$ to $\deg_\psi(\C^{(t)})$ may change at~$t_0$.
(More formally, we should write $F_{\RR^{(t)}}$ instead of $F_\RR$,
where $\RR^{(t)}$ is a rainbow $r$-partition of $\C^{(t)}$
whose combinatorial type does not depend on $t$. But we 
drop the superscript, keeping the dependence on $t$ implicit.)

A necessary condition for the change is that $F_{\RR}$ intersects $\psi$
just before or just after $t_0$. If it intersects $\psi$ \emph{both}
just before and just after $t_0$, then, using the genericity of $\psi$,
one can see that the geometric sign of $F_{\RR}$ does not
change, and so its contribution to the degree does not change either.
Thus, the only possibility is that $F_{\RR}$ intersects $\psi$
just before $t_0$ and does not intersect it just after, or the other way
round. 

By symmetry, it suffices to consider only the first case.
Let us also assume that $\sgn({\RR})=+1$ for $t<t_0$
(in some small open interval ending in $t_0$, that is).
Then, since $F_{\RR}$ stopped intersecting $\psi$ at $t_0$,
it must have passed $0$, and therefore, its geometric sign changed.
Thus, $\sgn({\RR})=-1$ for $t>t_0$, and the contribution
of $F_\RR$ to $\deg(\C^{(t)})$ has decreased by $1$ at~$t_0$.

Now we consider a permutation $\pi$ of $[r]$ and the
rainbow $r$-partition $\RR^\pi$, again depending on $t$.
By Lemma~\ref{l:permuting}, we have $\sgn({\RR^\pi})=
\sgn(\RR)$ all the time, so
$\sgn({\RR^\pi})$ also changes from $+1$ to $-1$ at $t_0$.
 Since the geometric sign of $F_{\RR^\pi}$ changes at $t_0$
(again by Lemma~\ref{l:permuting}), it means that $F_{\RR^\pi}$
passed through $0$ at $t_0$. So either it intersected $\psi$
just before $t_0$ and it does not intersect it just after,
or vice versa. In both cases, the contribution of $F_{\RR^\pi}$
to $\deg(\C^{(t)})$ has also decreased by $1$ at~$t_0$.

Since there are $r!$ choices for $\pi$, it follows that
the degree may change only by multiples of $r!$ as claimed.
The proposition is proved.
\end{proof}

\section{Computing the degree of a special BMZ-collection}

Here is the last step in the proof of Theorem~\ref{t:main}.

\begin{lemma}
\label{l:specific}
There is a  BMZ-collection $\C_0$ in sufficiently general position 
such that  $$|\deg \C_0|=((r-1)!)^{d+1}.$$
\end{lemma}

\begin{proof} The first $d+1$ color classes of $\C_0$ are
small clusters around the vertices of a regular $d$-dimensional 
simplex, as in Fig.~\ref{f:clusters}, and the single point $z$
of the last class is placed to the center of gravity of that simplex.

It is easy to see (and well known) that the Tverberg rainbow $r$-partitions
$\RR$ of $\C_0$ with $\RR\in\bR$ have $R_r=\{z\}$,
and the other $R_i$ each use
exactly one point of each $C_j$, $j=1,2,\ldots,d+1$. 
In the rook interpretation, they correspond to rook placements
where the $r$th row contains only the single rook in the last column,
and from this one immediately gets that their number is $((r-1)!)^{d+1}$.

It remains to see that all of these Tverberg $\RR$'s have the same sign.
It suffices to consider the effect of a local change, where we
swap two adjacent rows in one of the first $d+1$ chessboards (which
corresponds to moving some $c_j\in C_k$ from $R_i$ to $R_{i+1}$
and some $c_{j'}\in C_k$ from $R_{i+1}$ to $R_i$, $i+1\le r-1$).
This obviously changes the combinatorial sign. 

It remains to show that the geometric sign is also
changed by the swap. Let $\RR$
be the Tverberg $r$-partition before the swap and
$\RR^\leftrightarrow$ the one after the swap,
and let $M$  and $M^\leftrightarrow$ be the corresponding matrices
for $F_\RR$ and $F_{R^\leftrightarrow}$, as in the definition
of the geometric sign.
 Thus, the $j$th row
is $\varphi_i(c_j)$ in $M$ and   $\varphi_{i+1}(c_j)$ in $M^\leftrightarrow$, 
and the $j'$th row is $\varphi_{i+1}(c_{j'})$ in $M$ 
and $\varphi_{i}(c_{j'})$ in $M^\leftrightarrow$. 

Let $M'$ denote the matrix obtained from $M^\leftrightarrow$ 
by interchanging the
$j$th row with the $j'$th row. We have $\det(M')=-\det(M^\leftrightarrow)$,
and we want to check that $\sgn\det(M')=\sgn\det (M)$.

We can regard $M'$ as the matrix of vertex coordinates for the
$(N-1)$-dimensional simplex $F_\RR$ for a \emph{different}
BMZ-collection $\C_0'$, namely, the one obtained from $\C_0$
by interchanging $c_j$ with $c_{j'}$.  We prove a more general
statement:
whenever $\C_0'$ is a BMZ-collection obtained from $\C_0$ by moving
each of the points $c_j$ within its cluster arbitrarily
(and keeping $z$ fixed), then $\sgn\det(M')=\sgn\det(M)$.

It suffices to check that during a continuous
motion of some $c_j$ within its cluster, $\sgn \det(M)$ remains
constant. This sign may change only  when the
simplex $F_\RR$ becomes degenerate, or when the hyperplane
spanned by $F_\RR$ passes through $0$.

These two conditions translate, according to Lemma~\ref{l:dictio}, 
to the following:
during the continuous motion, the points of each class $R_i$, $i<r$, remain
affinely independent, and the $r$-partition $\RR-z$ 
never has either an affine Tverberg point or a Tverberg direction. 
The former holds because each $R_i$ has one point in each cluster. 
The latter holds trivially since the $r$th class of $\RR-z$ is empty.
This concludes the proof of Lemma~\ref{l:specific}.
\end{proof}

Now we have completed all steps from the proof outline, and thus
Theorem~\ref{t:main} is proved.

\section{Conclusion}\label{s:concl}

\heading{Configurations with degree 0.}
Suppose that there is an integer $r$ for which exists a BMZ-collection without
a Tverberg point. Then the degree of this collection has to be $0$, and thus
$r$ cannot be a prime number.

We performed computational experiments in the case $r = 4$,
with $d=2,3$.
We generated BMZ-collections at random inside the unit square
(or cube). We 
frequently obtained collections with degree $0$; however,
all of them had a Tverberg point. See Fig.~\ref{f:degzero42} for a
configuration with degree zero and few Tverberg partitions. 
We also obtained a collection of degree $0$
for $r=6$ and $d=2$. In this case the computation was already quite time
consuming (with our algorithm), and thus we performed only
a small number of experiments.
 
We believe that  BMZ-collections of degree $0$ exist for all non-prime $r$
and in all dimensions, but  unfortunately, we do not have a proof for this.

A natural idea for such a proof is to start with two  BMZ-collections
$\C_1$ and $\C_2$,
one of a positive degree and one of a negative degree,
and then transform $\C_1$ to $\C_2$ by a (generic) continuous motion.
\emph{If} we knew that the degree may jump only by $\pm r!$
during the motion, we would reach degree $0$ at some moment
(since the degree is always congruent to $((r-1)!)^d$
modulo $r!$, as we know, and $((r-1)!)^d$ is divisible by $r!$
for $d\ge 2$ and non-prime~$r$).
However, it turns out that even during a generic motion, there may be jumps
by larger multiples of $r!$, and so a subtler argument
is needed.

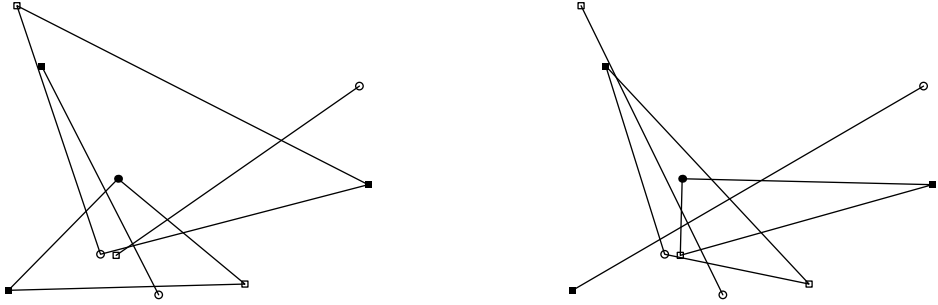
\labpdffig{degzero42}{A BMZ-collection with $r=4$ and $d=2$ of degree zero with
only two different Tverberg partitions (more precisely with only two
equivalence classes of $\sim$).}

\heading{A direct definition of sign?} A natural question is, whether
one can define the sign 
of a rainbow partition directly, without going through
the Sarkaria--Onn transform. However, it seems that if there
is such a direct definition (only referring to the mutual
position of the points of the rainbow partition) at all,
it has to be rather complicated. We will illustrate
this with an example concerning the simplest nontrivial
case, with  $d = 2$ and $r = 3$. 

Thus, we consider points $c_1,c_2,\ldots,c_6,z$ in the plane,
and the following BMZ-collection:
$C_1 = \{c_1, c_2\}$, $C_2 = \{c_3,
c_4\}$,$C_3 = \{c_5, c_6\}$, $C_4 = \{z\}$. 
We consider several rainbow partitions
$\RR \in \bR$ and the dependence of $\sgn(\RR)$ on the
positions of the~$c_i$. From the definition of the sign we get that $\sgn(\RR)
= 0$ iff at least on of the conditions of Lemma~\ref{l:dictio} holds. 
Hence it is plausible to assume that the sign changes when the
BMZ-collection moves over a position where $\RR - z$ has an affine Tverberg
point, or if one of the partition sets of $\RR - z$ is affinely dependent,
 or, finally, if $\RR - z$
has a Tverberg direction.
However, the movement must be sufficiently generic, otherwise the collection
could only ``reflect'' and the sign would not change. We did not attempt to
describe such a generic movement precisely since we are not aware of
convincing consequences\footnote{If there were a direct definition of sign using
this property then it would be surely of our interest.} (except for the discussion in this section).

%We moreless know that $\sgn(\RR)$ changes when
%$(C_1,C_2,C_3)$ has an affine Tverberg point; see Lemma~\ref{l:dictio}(i) (a bit more
%carefulness would be necessary in order to make the statement precise). Hence we
%have the following consequences.

First we set $R_1 := \{c_1, c_3, c_5\}$, $R_2 := \{c_2, c_4, c_6\}$, and
$R_3 := \{z\}$.
In this case $\sgn \RR = 0$ iff at least one of the triangles $c_1c_3c_5$ or
$c_2c_4c_6$ is degenerate. Thus a reasonable guess is that the sign depends
only on the cyclic orientations of the triangles $c_1c_3c_5$ and
$c_2c_4c_6$.\footnote{As we pointed out above, we do not have a precise proof.
However, our observation is also supported by a computer program for computing
the sign on many examples. A similar remark also applies 
for other choices of~$\RR$.}

For $R_1 := \{c_1, c_3, c_6\}$, $R_2 := \{c_2, c_4\}$, $R_3 := \{c_5,
z\}$, the situation is similar. The sign depends only on the cyclic
orientation of the triangles $c_1c_3c_6$ and $c_2c_4c_5$.

Finally, let  $R_1 := \{c_2, c_5\}$, $R_2 := \{c_4, c_6\}$, $R_3 
:= \{c_1, c_3, z\}$.
Then the sign depends on the orientation of the lines $c_2c_5$, $c_4c_6$ and
$c_1c_3$. However, it also depends on the mutual position of these lines,
and it changes when all three of them pass through a common point. See
Figure~\ref{f:moving}.

Unfortunately, we are not aware of a simple uniform description of the three
cases above.

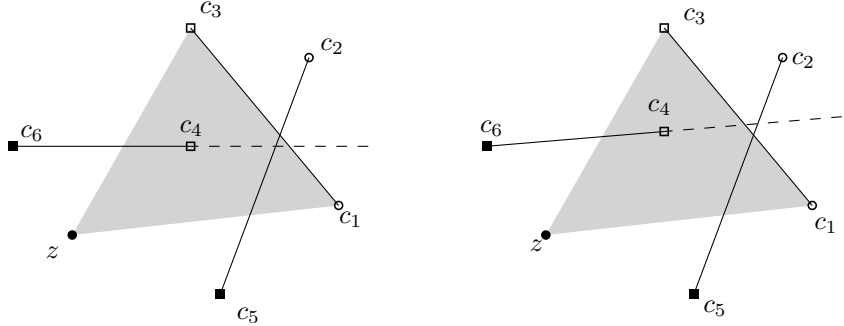
\labpdffig{moving}{The degree of this partition changes when the three lines
pass through a common point.}

\subsection*{Acknowledgement}

We would like to thank Marek Kr\v{c}\'al for useful discussions
at initial stages of this research. We also thank G\"unter M.~Ziegler for valuable comments,
and Peter Landweber and two anonymous referees for detailed comments and corrections
that greatly helped to improve the presentation. In particular, we are indebted to one of the 
referees for pointing out to us the reference \cite{Zivaljevic-InPursuitOfTheColoredCaratheodoryBaranyTheorems-1995}.

\bibliographystyle{alpha}
\bibliography{comb-ct}

\begin{thebibliography}{BLV{\v{Z}}94}

\bibitem[BB79]{BajmoczyBarany}
E.~G. Bajm{\'o}czy and I.~B{\'a}r{\'a}ny.
\newblock A common generalization of {B}orsuk's and {R}adon's theorem.
\newblock {\em Acta Math.\ Hungarica}, 34:347--350, 1979.

\bibitem[BFL90]{barany-furedi-lovasz90}
I.~B{\'a}r{\'a}ny, Z.~F{\"u}redi, and L.~Lov{\'a}sz.
\newblock On the number of halving planes.
\newblock {\em Combinatorica}, 10(2):175--183, 1990.

\bibitem[BL92]{barany-larman92}
I.~B{\'a}r{\'a}ny and D.~G. Larman.
\newblock A colored version of {T}verberg's theorem.
\newblock {\em J. London Math. Soc. (2)}, 45(2):314--320, 1992.

\bibitem[BLV{\v{Z}}94]{bjorner-at-al94}
A.~Bj{\"o}rner, L.~Lov{\'a}sz, S.~T. {V}re{\'c}ica, and R.~T.
  {\v{Z}}ivaljevi{\'c}.
\newblock Chessboard complexes and matching complexes.
\newblock {\em J. London Math. Soc. (2)}, 49(1):25--39, 1994.

\bibitem[BMZ09]{blagojevic-matschke-ziegler09arxiv}
P.~V.~M. Blagojevi\'{c}, B.~Matschke, and G.~M. Ziegler.
\newblock Optimal bounds for the colored {T}verberg problem.
\newblock Preprint; http://arxiv.org/abs/0910.4987, 2009.

\bibitem[BMZ11]{BMZ:OptimalBoundsTverbergVrecica-2011}
P.~V.~M. Blagojevi\'{c}, B.~Matschke, and G.~M. Ziegler.
\newblock Optimal bounds for a colorful {T}verberg--{V}re\'cica type problem.
\newblock {\em Adv.~Math.}, 226:5198--5215, 2011.

\bibitem[BSS81]{BaranySS}
I.~B\'{a}r\'{a}ny, S.~B. Shlosman, and A.~Sz{\H{u}}cs.
\newblock On a topological generalization of a theorem of {T}verberg.
\newblock {\em J. London Math. Soc., II. Ser.}, 23:158--164, 1981.

\bibitem[Gr{\"u}03]{Grunbaum:ConvexPolytopes2nd-2003}
Branko Gr{\"u}nbaum.
\newblock {\em Convex polytopes}, volume 221 of {\em Graduate Texts in
  Mathematics}.
\newblock Springer-Verlag, New York, second edition, 2003.

\bibitem[Kal95]{Kalai:CombinatoricsConvexity-95}
Gil Kalai.
\newblock Combinatorics and convexity.
\newblock In {\em Proceedings of the International Congress of Mathematicians,
  Vol.\ 1, 2 (Z\"urich, 1994)}, pages 1363--1374, Basel, 1995. Birkh\"auser.

\bibitem[Mat96]{matousek96}
J.~Matou{\v{s}}ek.
\newblock Note on the colored {T}verberg theorem.
\newblock {\em J. Combin. Theory Ser. B}, 66(1):146--151, 1996.

\bibitem[Mat02]{Mat-dg}
J.~Matou\v{s}ek.
\newblock {\em Lectures on Discrete Geometry}.
\newblock Springer, New York, 2002.

\bibitem[Mat04]{Mat-kne}
J.~Matou\v{s}ek.
\newblock A combinatorial proof of {K}neser's conjecture.
\newblock {\em Combinatorica}, 24(1):163--170, 2004.

\bibitem[Sar91]{Sarkaria-flores}
K.~S. Sarkaria.
\newblock A generalized van {Kampen--Flores} theorem.
\newblock {\em Proc. Amer. Math. Soc.}, 111:559--565, 1991.

\bibitem[Sar92]{sarkaria92}
K.~S. Sarkaria.
\newblock {T}verberg's theorem via number fields.
\newblock {\em Israel J. Math.}, 79:317--320, 1992.

\bibitem[Tve66]{Tverberg}
H.~Tverberg.
\newblock A generalization of {R}adon's theorem.
\newblock {\em J. London Math. Soc.}, 41:123--128, 1966.

\bibitem[Tve81]{Tverberg2}
H.~Tverberg.
\newblock A generalization of {R}adon's theorem. {II}.
\newblock {\em Bull. Aust. Math. Soc.}, 24:321--325, 1981.

\bibitem[V{\v{Z}}09]{vrecica-zivaljevic09arxiv}
S.~T. Vre\'cica and R.~T. {\v{Z}}ivaljevi\'c.
\newblock Chessboard complexes indomitable.
\newblock Preprint; http://arxiv.org/abs/0911.3512, 2009.

\bibitem[Zie02]{Ziegler-kne}
G.~M. Ziegler.
\newblock Generalized {K}neser coloring theorems with combinatorial proofs.
\newblock {\em Invent. Math.}, 147:671--691, 2002.
\newblock Erratum {\it ibid.}, 163:227--228, 2006.

\bibitem[{\v{Z}}iv95]{Zivaljevic-InPursuitOfTheColoredCaratheodoryBaranyTheore%
ms-1995}
Rade~T. {\v{Z}}ivaljevi{{\'c}}.
\newblock In pursuit of colored {C}arath{\'e}odory-{B}{\'a}r{\'a}ny theorems.
\newblock {\em Publ. Inst. Math. (Beograd) (N.S.)}, 57(71):91--100, 1995.

\bibitem[{\v{Z}}iv98]{zival-ugII}
R.~T. {\v{Z}}ivaljevi{\'c}.
\newblock User's guide to equivariant methods in combinatorics. {I}{I}.
\newblock {\em Publ. Inst. Math. (Beograd) (N.S.)}, 64(78):107--132, 1998.

\bibitem[{\v{Z}}V92]{zivaljevic-vrecica92}
R.~T. {\v{Z}}ivaljevi{\'c} and S.~T. Vre{\'c}ica.
\newblock The colored {T}verberg's problem and complexes of injective
  functions.
\newblock {\em J. Combin. Theory Ser. A}, 61(2):309--318, 1992.

\end{thebibliography}

\end{document}